\documentclass[a4paper]{amsart}

\usepackage{cite}

\usepackage{color}

\newcommand{\Bk}{\color{black}}

\newcommand{\dd}{\mathrm{d}}

\DeclareMathOperator{\tr}{tr}

\newcommand{\ZZ}{\mathbb{Z}}
\newcommand{\RR}{\mathbb{R}}

\newcommand{\NN}{\mathbb{N}}
\newcommand{\TT}{\mathbb{T}}
\newcommand{\SSS}{\mathbb{S}}

\newcommand{\cO}{\mathcal{O}}
\newcommand{\cD}{\mathcal{D}}
\newcommand{\cC}{\mathcal{C}}
\newcommand{\cR}{\mathcal{R}}

\newcommand{\cN}{\mathcal{N}}

\newtheorem{theorem}{Theorem}
\newtheorem{prop}[theorem]{Proposition}

\newtheorem{lemma}[theorem]{Lemma}

\theoremstyle{definition}

\newtheorem{rem}[theorem]{Remark}

\begin{document}

\title[Eigenvalue counting function on~conical domains]{Eigenvalue counting function for  Robin Laplacians on~conical domains}

\author{Vincent Bruneau}

\address{Institut math\'ematique de Bordeaux, Universit\'e de Bordeaux,
351 cours de la lib\'eration, 33405 Talence Cedex, France }

\email{vincent.bruneau@math.u-bordeaux1.fr}
\urladdr{http://www.math.u-bordeaux1.fr/~vbruneau/}

\author{Konstantin Pankrashkin}
\address{Laboratoire de Math\'ematiques d'Orsay, Univ.~Paris-Sud, CNRS, Universit\'e Paris-Saclay, 91405 Orsay, France}

\email{konstantin.pankrashkin@math.u-psud.fr}
\urladdr{http://www.math.u-psud.fr/~pankrash/}

\author{Nicolas Popoff}

\address{Institut math\'ematique de Bordeaux, Universit\'e de Bordeaux,
351 cours de la lib\'eration, 33405 Talence Cedex, France }

\email{nicolas.popoff@math.u-bordeaux1.fr}
\urladdr{http://www.math.u-bordeaux1.fr/~npopoff/}

\begin{abstract}
We study the discrete spectrum of the Robin Laplacian $Q^{\Omega}_\alpha$
in $L^2(\Omega)$, 
\[
u\mapsto -\Delta u, \quad \dfrac{\partial u}{\partial n}=\alpha u \text{ on }\partial\Omega,
\]
where $\Omega\subset \RR^{3}$ is a conical domain with a regular cross-section $\Theta\subset \SSS^2$,
$n$ is the outer unit normal, and $\alpha>0$ is a fixed constant.
It is known from previous papers that the bottom of the essential spectrum of $Q^{\Omega}_\alpha$ is $-\alpha^2$ and that the finiteness
of the discrete spectrum depends on the geometry of the cross-section.
We show that the accumulation
of the discrete spectrum of $Q^\Omega_\alpha$
is determined by the discrete spectrum of an effective Hamiltonian defined on the boundary and far from the origin. By studying this model operator, we prove that the number of eigenvalues of $Q^{\Omega}_\alpha$ in $(-\infty,-\alpha^2-\lambda)$, with $\lambda>0$, behaves for $\lambda\to0$ as
\[
\dfrac{\alpha^2}{8\pi \lambda} \int_{\partial\Theta} \kappa_+(s)^2\dd s +o\left(\frac{1}{\lambda}\right),
\]
where $\kappa_+$ is the positive part of the geodesic curvature of the cross-section boundary.
\end{abstract}



\maketitle

\section{Introduction}

For an open set $\Omega\subset\RR^d$ and with a suitably regular boundary and a constant $\alpha>0$, define the associated Robin Laplacian
$Q^\Omega_\alpha$ as the unique self-adjoint operator associated with the quadratic form
\[
q^\Omega_\alpha(u,u)=\int_{\Omega} |\nabla u|^2\dd x -\alpha\int_{\partial\Omega} u^2\dd \sigma,
\quad \cD(q^\Omega_\alpha)=H^1(\Omega),
\]
where $\sigma$ stands for the $(d-1)$-dimensional Hausdorff measure on $\partial\Omega$.
Informally, the operator $Q^\Omega_\alpha$ acts as $u\mapsto -\Delta u$
on the functions $u$ satisfying the Robin boundary conditions $D_n u=\alpha u$,
where $D_n$ is the outer unit normal derivative. Numerous recent works
have studied various links between the geometric characteristics of $\Omega$ and the spectral properties
of $Q^{\Omega}_\alpha$ for large $\alpha$, see e.g. \cite{BruPof16,hk,kkr,pp14,pp15}.
In contrast to these works, in the present paper we deal
with the case of a \emph{fixed} $\alpha$, and we concentrate our attention on the case
when $\Omega$ is a conical domain defined as follows.

Let $\SSS^2$ denote the two-dimensional unit sphere and $\Theta\subset\SSS^2$ be a Lipschitz domain.
The set
\[
\Lambda(\Theta):=\big\{
x=r\theta: r\in \RR_+, \  \theta\in\Theta
\big\}\subset\RR^3
\]
will be called the cone with the cross-section $\Theta$.
By a conical domain with the cross-section $\Theta$ we mean
any Lipschitz domain coinciding with $\Lambda(\Theta)$ outside a ball.

In order to state the results, we need to recall some geometric notions. Let $\gamma\subset \SSS^2$ be a $C^2$ loop of length $\ell$
and let $n:\gamma\to \SSS^2$ be a continuous vector field tangent to $\SSS^2$ and orthogonal to $\gamma$.
Let $\Gamma\in \RR/(\ell\ZZ)\to \gamma\subset \SSS^2$
be an arc-length parametrization of $\gamma$ such that the vector $\Gamma\times \Gamma'$ coincides with $n$,
then the geodesic curvature  $\kappa(M)$ of $\gamma$ with respect to $n$ at the point $M=\Gamma(s)$
is the mixed product
\[
\kappa(M)=\big[\Gamma''(s),\Gamma'(s),\Gamma(s)\big].
\]

Throughout the paper we will assume that $\Omega \subset\RR^3\Bk$ is a conical domain
whose cross-section $\Theta$ is $C^4$ smooth, and we denote
\[
\kappa:\partial\Theta\to \RR
\]
the geodesic curvature of the boundary of the cross-section with respect to the outer normal.
Recall that a connected $\Theta$ is geodesically convex iff
the geodesic curvature of its boundary with respect to the outer normal is non-negative,
see e.g. \cite[Proposition 2.1]{bl}.
The following result was obtained in earlier works, see Theorem~2.1 and Section~5 in \cite{mmnp}:

\begin{prop}\label{prop1} The essential spectrum of $Q^\Omega_\alpha$ equals
$[-\alpha^2,+\infty)$. If $\Theta$ is the complement of a geodesically
convex subset of $\SSS^2$,
then the discrete spectrum of $Q^\Omega_\alpha$
is finite. If $\kappa$ is strictly positive in at least one point, then
the discrete spectrum of $Q^\Omega_\alpha$ is infinite.
\end{prop}
In the present paper, we are going to arrive at a more detailed
information on the infinite discrete spectrum.
For a bounded from below self-adjoint operator $A$  having eigenvalues $E_0(A)\leq E_1(A) \leq \cdots \leq E_j(A) \leq  \cdots $,  below the bottom of its essential spectrum  and for $\lambda\in \RR$ below the essential spectrum,
we define the counting function by
\[
\cN(A,\lambda):=\tr {\mathbf{1}}_{(-\infty,\lambda]}(A) = \# \{ j \in \NN; \; E_j(A) \leq \lambda \}.
\]

First we show the existence of an effective Hamiltonian, defined on the boundary and far from the origin,
which provides the number of eigenvalues below the bottom of the essential spectrum, up to a finite number.
For shortness, we will study the normalized operator
\[
Q^\Omega:=Q^\Omega_1.
\] 
 The case of an arbitrary $\alpha>0$ is then easily included using the fact that
\begin{equation}
 \label{eq-ej}
E_j(Q^\Omega_\alpha)=\alpha^2 E_j(Q^{\alpha\Omega})
\end{equation}
and that $\alpha\Omega$ is a conical
domain with the same cross-section.

\begin{theorem}\label{thm2}
\label{T:approx}
There exist continuous functions $a_{\pm},b_{\pm},\nu_{\pm} : \RR_+ \to \RR$ with
\begin{gather*}
\inf_{r \in \RR_+} a_{\pm}(r) >0, \quad \inf_{r \in \RR_+} b_{\pm}(r) >0 ,\nonumber
\\
\lim_{r \rightarrow + \infty} a_{\pm}(r)=
\lim_{r \rightarrow + \infty} b_{\pm}(r)=1, \quad  \lim_{r \rightarrow + \infty} \nu_{\pm}(r)=0,
\end{gather*}
such that for some $R>0$ and $M_{R}>0$ there holds 
\[
\cN(K_+,-\lambda)-M_R\leq \cN(Q^{\Omega},-1-\lambda)\le \cN(K_-,-\lambda)+M_R \text{ for any } \lambda>0,
\] \Bk
where $K_{\pm}$ are the self-adjoint operators acting in $L^2(U)$,
\[
U=(R,+\infty)\times \partial\Theta,
\]
and associated with the quadratic forms 
\[
k_{\pm}(v)=\int_R^\infty \int_{\partial\Theta} a_{\pm}(r)v_r^2
+\dfrac{b_{\pm}(r)}{r^2}\, v_s^2 -\dfrac{\kappa(s)+\nu_{\pm}(r)}{r}\, v^2 \, \dd s\, \dd r
\]
defined on $\cD(k_{-})=H^1(U)$ and $\cD(k_{+})=H^1_0(U)$. 
In the above formula and in the rest of the paper, $v_s$ and $\dd s$ mean respectively
the derivative of $v$ and the integration with respect to the arc-length, while $v_r$ denotes the derivative of $v$ with respect to $r$.
\end{theorem}
This approximation result is proved in sections \ref{S:Bracketing}--\ref{S:lowerbound} by studying the quadratic form in suitable tubular coordinates far from the origin. Then we study in Section \ref{S:Weylmodel} the number of negative eigenvalues of the effective Hamiltonian, and we deduce our main result:
\begin{theorem}\label{thm3}
For $\lambda\to 0+$ there holds
\begin{equation}
       \label{eq-weyl}
\cN(Q^\Omega_\alpha,-\alpha^2-\lambda)=\dfrac{\alpha^2}{8\pi \lambda} \int_{\partial\Theta}\kappa_+^2\dd s
+o\Big(\dfrac{1}{\lambda}\Big),
\end{equation}
where $\kappa_+(s):=\max\big\{\kappa(s),0\big\}$ and $\dd s$ means the integration with respect to the arc-length.
\end{theorem}
We remark that if the cross-section $\Theta$ is the complement of a geodesically convex set, then $\kappa\le 0$
and the first term on the right-hand side of \eqref{eq-weyl} is zero so we arrive at
$\cN(Q^\Omega,-1-\lambda)=o(\lambda^{-1})$ for $\lambda\to 0+$. In fact, the limit is finite by Proposition~\ref{prop1}.

More general cones are considered on \cite{BruPof16}, and the bottom of the essential spectrum is given in \cite[Theorem 1.2]{BruPof16}. The existence and finiteness of the discrete spectrum are more complicated in the general situation. In particular, as shown in \cite{mmnp}, a convex
cone with a non-smooth cross-section can have a finite discrete spectrum.

Note that the question of the distribution of the discrete spectrum below the bottom of the essential spectrum is already treated in numerous situations. For the Schr\"odinger operators $-\Delta+V$ in $\RR^d$, several results have known since a long time for short range or long range perturbations $V$ (see Section XIII.15 of \cite{ReSi78}). However, the recent work \cite{Rai15b} gives news results for oscillating decaying potentials. These Schr\"odinger operators are also used as effective Hamiltonians for more complex problems, see \cite{BerExLot14,LotOur16} for $\delta$-interaction supported on conical domains, \cite{DauOurRay, ExnTat} for Dirichlet Laplacian in a conical layer, \cite{BrKoRaSo09, Rai15}  for twisted waves guides, \cite{BriRaiSoc08} for magnetic Hamiltonian on a strip, \cite{BruMirRai14}  for magnetic Hamiltonian on the half plan with Neumann condition, \cite{MirRai12} for a model of quantum Hall effect, or \cite{HS} for Schr\"odinger operators with degenerate kinetic terms. 
Another family of results concerns perturbations of magnetic hamiltonians where the accumulation is governed by some Toeplitz operators (see \cite{Rai90, RaWa02, FeRa04, Rai06, Rai10}) or by Anti-Wick pseudo-differential operators as in \cite{BruMirRai11, BruMirRai14, LunRai15}.
The recent paper \cite{kkr} studied the Weyl asymptotics for the negative discrete spectrum of $Q^\Omega_\alpha$
for bounded smooth $\Omega$ and large $\alpha$, but the nature of the asymptotics is completely different.

Here, as shown by Theorem  \ref{T:approx}, our problem involves a new sort of model operators for which we prove a semi-classical type asymptotics.
We also remark that the spherical coordinate approach can be adapted to other conical configurations, see \cite{top}.

\section{Bracketing for the counting functions and tubular coordinates}
\label{S:Bracketing}
\subsection{Cutting out the vertex}

For shortness, we denote
\[
Q:=Q^\Omega.
\]
For $R>0$ we denote $B_R:=\big\{x\in \RR^3: |x|<R\big\}$ and
pick any  $R>1$ such that the domain
\[
\Omega_R:=\Omega\setminus \overline{B_R}
\]
 coincides
with the portion of the respective cone
\[
\Lambda_R:=\{r \theta: r\in(R,+\infty),\, \theta\in \Theta\}.
\]
Additional conditions on $R$ will be formulated during the subsequent computations.

Consider the operators $Q^{N/D}_R$ in $L^2(\Lambda_R)$ given by the quadratic forms
\begin{gather*}
q^N_R(u)=\int_{\Lambda_R} |\nabla u|^2\dd x-\int_{\Lambda_R\cap\partial\Omega} u^2\dd \sigma,
\quad \cD(Q^N_R)=H^1(\Lambda_R),\\
q^D_R:= \text{the restriction of $q^N_R$ to } \cD(q^D_R):=\big\{u\in H^1(\Lambda_R):\, u=0 \text{ on } \partial \Lambda_R \setminus\partial\Omega\big\},
\end{gather*}
and, furthermore, the operators $A^{N/D}_R$ in $L^2(\Omega\cap B_R)$ generated by the quadratic forms
\begin{gather*}
a^N_R(u)=\int_{\Omega\cap B_R} |\nabla u|^2\dd x-\int_{B_R\cap \partial\Omega} u^2\dd\sigma, \quad
\cD(a^N_R)=H^1(\Omega\cap B_R),\\
a^D_R:= \text{the restriction of $a^N_R$ to } \cD(a^D_R):=\big\{u\in H^1(\Omega\cap B_R):\, u=0 \text{ on } \Omega\cap \partial B_R \big\}.
\end{gather*}
By the min-max principle one has, for any $j\in \NN$,
\[
E_j(Q^N_R\oplus A^N_R)\le E_j(Q^\Omega)\le E_j(Q^D_R\oplus A^D_R).
\]
Remark that the operators $A^N_R$ and $A^D_R$ have compact resolvents
and have at most finitely many eigenvalues in $(-\infty,-1)$.
It follows that for any $R>0$ there exists $M_R>0$ such that
\begin{equation}
    \label{eq-nn1}
 \cN(Q^D_R,-1-\lambda)-M_R\le \cN(Q^{\Omega},-1-\lambda)\le \cN(Q^N_R,-1-\lambda)+M_R, \quad \lambda>0.
\end{equation}
Therefore, it is sufficient to study the accumulation rate for the eigenvalues
of  $Q^{N/D}_R$ with any fixed $R>0$.

\subsection{Reduction to a neighborhood of the boundary}
We introduce a function $\delta:(0,+\infty)\to (0,+\infty)$ of the form 
\begin{equation}
          \label{eq-delta}
\delta(r)=c r^{-\rho}, \quad c>0, \, \rho >0. 
\end{equation}
Denote $\gamma:=\partial\Theta$, and consider the domains
\[
\Lambda_{R,\delta}:= \Big\{
x\in \Lambda_R:  d\left(\dfrac{x}{|x|},\gamma\right)< \delta\big(|x|\big)
\Big\}.
\]
where $d$ stands for the geodesic distance on the unit sphere $\SSS^2$. In addition,
for $a>0$ we shall denote
\[
\Theta_a:=\big\{\theta\in \Theta: d(\theta,\gamma)< a\big\},
\]
then one can rewrite
\[
\Lambda_{R,\delta}= \Big\{
x\in \Omega_R: \dfrac{x}{|x|} \in \Theta_{\delta(|x|)}
\Big\}.
\]
Let $Q^N_{R,\delta}$ and $Q^D_{R,\delta}$  denote the self-adjoint operators
acting in $L^2(\Lambda_{R,\delta})$ and generated by the quadratic forms
\begin{gather*}
q^N_{R,\delta}(u)=\int_{\Lambda_{R,\delta}} |\nabla u|^2\dd x-\int_{\Lambda_{R,\delta}\cap\partial\Omega} u^2\dd \sigma,
\quad \cD(q^N_{R,\delta})=H^1(\Lambda_{R,\delta}),\\
q^D_{R,\delta}:= q^N_{R,\delta} \text{ restricted to }
\cD(q^D_{R,\delta}):=\big\{u\in H^1(\Lambda_{R,\delta}):\, u=0 \text{ on } \partial \Lambda_{R,\delta} \setminus\partial\Omega\big\},
\end{gather*}
and by $C_{R,\delta}^N$ and $C_{R,\delta}^D$ we denote respectively the Neumann
and the Dirichlet Laplacians in $L^2(\Lambda_R\setminus \overline{\Lambda_{R,\delta}})$.
By the min-max principle we have, for any $j\in \NN$,
\[
E_j(Q^N_{R,\delta}\oplus C^N_{R,\delta})
\le
E_j(Q^N_R)
\le
E_j(Q^D_{R,\delta}\oplus C^D_{R,\delta}).
\]
As both $C^N_{R,\delta}$ and $C^N_{R,\delta}$ are non-negative, we have, for any  $\lambda>0$,
\begin{equation}
         \label{eq-nn2}
\cN(Q^D_{R,\delta},-1-\lambda)\le \cN(Q^D_R,-1-\lambda)\le \cN(Q^N_R,-1-\lambda)\le  \cN(Q^N_{R,\delta},-1-\lambda).
\end{equation}

\subsection{Quadratic forms in spherical coordinates}

To have more explicit expressions, let us pass to the spherical coordinates.
Denote
\[
\Sigma:=\big\{(r,\theta): r\in (R,+\infty), \, \theta \in \Theta_{\delta(r)}
\big\}
\]
and consider the unitary transform
\[
V_1: L^2(\Lambda_{R,\delta})\to L^2( \Sigma, \dd r\dd \theta), \quad
(V_1 u)(r,\theta)=ru(r\theta),
\]
and the quadratic forms
\[
h^{D/N}(u)=q^{D/N}_{R,\delta} (V_1^{-1}u), \quad \cD(h^{D/N})=V_1 \cD(Q^{N/D}_{R,\delta}).
\]
The standard computation gives
\[
h^{D/N}(u)=\int_R^\infty \bigg[\int_{\Omega_{\delta(r)}} \Big| \dfrac{\partial u}{\partial r}(r,\cdot)\Big|^2 \dd \theta
+\dfrac{1}{r^2} \bigg(\int_{\Omega_{\delta(r)}} |\nabla_\theta u|^2\dd \theta - r\int_\gamma u^2\dd s \bigg)\bigg]\dd r,
\]
where $s$ is the arc-length on $\gamma$ and
\[
\cD(h^N)=H^1(\Sigma),
\quad
\cD(h^D)=\big\{u\in H^1(\Sigma): u(r,\cdot)=0,\, u\big(\cdot,\delta(\cdot)\big)=0\big\}.
\]
Remark that, by construction, the self-adjoint operators $H^N$ and $H^D$
generated by $h^N$ and $h^D$ respectively and acting in $L^2(\Sigma)$
are unitarily equivalent to $Q^N_{R,\delta}$ and $Q^D_{R,\delta}$ respectively
and, hence, have the same counting functions. The combination with \eqref{eq-nn1}
and \eqref{eq-nn2} leads to the following summary of the preceding considerations:
\begin{prop}
For any $R>0$ there exist $M_R>0$ such that
\[
\cN(H^D,-1-\lambda)-M_R\le \cN(Q,-1-\lambda)\le  \cN(H^N,-1-\lambda), \quad \lambda>0.
\]
\end{prop}

\subsection{Tubular coordinates $(r,s,t)$}

By assumption, the boundary $\partial\Theta$ consists of $m$ closed loops, to be denoted by $\gamma_j$, $j=1,\dots, m$.
Denote
\[
\ell_j:=\text{length of $\gamma_j$}, \quad \TT:=\RR/(\ell_1\ZZ) \sqcup \dots \sqcup \RR/(\ell_m\ZZ),
\]
and let
\[
\Gamma:\TT\to \SSS^2\subset\RR^3
\]
be an arc-length parametrization of $\partial\Theta$ oriented in such a way that
the outer (with respect to $\Theta$) unit normal to $\partial\Theta$ is given by
\[
n:=\Gamma\times\Gamma'.
\]
Consider the map
\[
\TT\times(0,a)\ni(s,t)\mapsto 
\phi(s,t):=(\cos t)\, \Gamma(s)-(\sin t )\, n(s)\in \SSS^2\subset \RR^3.
\]
It is a simple geometric fact that one can find a sufficiently small $a>0$ such that
\[
d\big(\phi(s,t),\gamma)\equiv t
\]
and $\phi$ is a diffeomorphism between $\TT\times(0,a)$ and $\Theta_a$.
Let us calculate the associated metric tensor.
Note that $\Gamma$, $\Gamma'$ and $n$ form an orthonormal basis. Moreover, by derivating $\|n\|^2=1$ and $n\cdot \Gamma=0$, we obtain that $n'$ is orthogonal to $n$ and to $\Gamma$ and then $n'(s)= \kappa(s) \Gamma'(s) $. Thus, we have
\begin{align*}
\dfrac{\partial \phi(s,t)}{\partial s}&=\cos t \, \Gamma'(s)-\sin t \, n'(s)= (\cos t - \sin t \, \kappa(s))\Gamma'(s) ,\\
\dfrac{\partial \phi(s,t)}{\partial t}&=-\sin t \, \Gamma(s)-\cos t \, n(s).
\end{align*}
We deduce
\begin{align*}
\dfrac{\partial \phi(s,t)}{\partial s}\cdot \dfrac{\partial \phi(s,t)}{\partial s} & =w(s,t)^2 \ \ \mbox{with} \ w(s,t)=\cos t-\sin t\, \kappa(s) \\
\dfrac{\partial \phi(s,t)}{\partial s}\cdot \dfrac{\partial \phi(s,t)}{\partial t}& =0, \\
\dfrac{\partial \phi(s,t)}{\partial t}\cdot \dfrac{\partial \phi(s,t)}{\partial t}& =1 \, \,.
\end{align*}
Therefore, the metric tensor $G$ and the volume form $g$ associated with $\phi$ are
\[
G=\begin{pmatrix}
w(s,t)^2 & 0 \\
0 &1
\end{pmatrix},
\quad
g=\sqrt{\det G}= |w(s,t)|.
\]
Without loss of generality we assume that $a<1/\|\kappa\|_\infty$, so that $w>0$, and then
\[
G^{-1}=\begin{pmatrix}
w(s,t)^{-2} & 0 \\
0 & 1
\end{pmatrix}.
\]
Furthermore, we assume that the constant $c$ in \eqref{eq-delta} is sufficiently
small to have $\delta<a$ for all $r>R$ and denote
\[
\Pi_r:=\TT\times \big(0,\delta(r)\big), \quad r>R,
\quad
\Pi:=\big\{ (r,s,t): r\in (R,+\infty),\, (s,t)\in \Pi_r\big\}.
\]
The above diffeomorphism $\phi$
produces a unitary transform
\[
V_2:L^2(\Sigma)\to L^2\Big(\Pi,w\dd s\,\dd t\Big),
\quad
(V_2 u)(r,s,t)=u\big(r,\phi(s,t)\big)
\]
and the quadratic forms $p_0^{N/D}:=h^{N/D}\circ V_2^{-1}$
are given by
\begin{multline*}
p_0^{N/D}(u)=\int_R^{+\infty}\Bigg( \int_{\Pi_r} w(s,t) \Big(\dfrac{\partial u}{\partial r}\Big)^2\, 
\,\dd s\,\dd  t\\
+\dfrac{1}{r^2}
\int_{\Pi_r} w(s,t)^{-1} \Big(\dfrac{\partial u}{\partial s}\Big)^2
+w(s,t)\Big(\dfrac{\partial u}{\partial t}\Big)^2 \dd s \dd t
-r\int_\TT u(r,s,0)^2\dd s \Bigg)\dd r
\end{multline*}
on the domains
\[
\cD(p_0^N)=H^1(\Pi),
\quad
\cD(p_0^D)=\big\{ u\in H^1(\Pi): u_{\partial}(\cdot,s)=0, \ u(R,\cdot,\cdot)=0\big\},
\]
where $u_\partial(r,s):=u\big(r,s,\delta(r)\big)$.
To remove the weight we apply another unitary transform
\[
V_{3}:L^2(\Pi,w\dd s\,\dd t)\to L^2(\Pi,\dd s\dd t), \quad
(V_3 u)(r,s,t)=w(s,t)^{\frac{1}{2}}u(r,s,t).
\]
then the new quadratic forms $P^{N/D}:=p_0^{N/D}\circ V_3^{-1}$ are defined on the domains $\cD(p^D)=\big\{u\in H^1(\Pi):
u_\partial(\cdot,s)=0, \ u(R,\cdot,\cdot)=0\big\}$ and $\cD(p^N)=H^{1}(\Pi)$ by
\begin{multline*}
p^D(u)=
\int_R^{+\infty}\Bigg( \int_{\Pi_r}  \Big(\dfrac{\partial u}{\partial r}\Big)^2\, \dd s\,\dd  t\\
+\dfrac{1}{r^2} \bigg[
\int_{\Pi_r} \bigg(w(s,t)^{-2} \Big( \dfrac{\partial u}{\partial s}\Big)^2
+\Big( \dfrac{\partial u}{\partial t}\Big)^2 + K(s,t) u^2 \bigg) \dd t\, \dd s\\
-\int_\TT \Big( r+ \dfrac{\kappa(s)}{2}\Big)u(r,s,0)^2\dd s\bigg]\Bigg)\dd r
\end{multline*}
and
\begin{multline*}
p^N(u)=
p^D(u)+\int_{R}^{+\infty}\frac{1}{r^2}\int_\TT B(r,s) u\big(r,s,\delta(r)\big)^2\dd s\dd r, \quad \cD(p^N)=H^1(\Pi),
\end{multline*}
with $K$ and $B$ two continuous functions, uniformly bounded, respectively defined on $\TT\times (0,a)$ and $(R,+\infty)\times \TT$.
In particular, since $w$ is regular and  $w(s,0)=1$, there exist positive constants $C_{G},C_{K},C_{B}$ such that
\begin{multline*}
p^N(u)\ge p_-(u):=
\int_R^{+\infty}\Bigg( \int_{\Pi_r}  \Big(\dfrac{\partial u}{\partial r}\Big)^2\, \dd s\,\dd  t\\
+\dfrac{1}{r^2} \bigg[
\int_{\Pi_r} \bigg( (1-C_{G}\delta) \Big( \dfrac{\partial u}{\partial s}\Big)^2
+\Big( \dfrac{\partial u}{\partial t}\Big)^2 -C_{K} u^2 \bigg) \dd t\, \dd s\\
-\int_\TT \Big( r+ \dfrac{\kappa(s)}{2}\Big)u(r,s,0)^2\dd s 
- C_{B}\int_\TT u\big(r,s,\delta(r)\big)^2\dd s
 \bigg]\Bigg)\dd r, \quad \cD(p_-)=\cD(p^{N}),
\end{multline*}
and
\begin{multline*}
p^D(u)\le p_+(u):=
\int_R^{+\infty}\Bigg( \int_{\Pi_r}  \Big(\dfrac{\partial u}{\partial r}\Big)^2\, \dd s\,\dd  t\\
+\dfrac{1}{r^2} \bigg[
\int_{\Pi_r} \bigg( (1+C_{G}\delta) \Big( \dfrac{\partial u}{\partial s}\Big)^2
+\Big( \dfrac{\partial u}{\partial t}\Big)^2 +C_{K} u^2 \bigg) \dd t\, \dd s\\
-\int_\TT \Big( r+ \dfrac{\kappa(s)}{2}\Big)u(r,s,0)^2\dd s 
 \bigg]\Bigg)\dd r, \quad \cD(p_+)=\cD(p^{D}).
\end{multline*}
Denote by $P_\pm$ the self-adjoint operators acting in $L^2(\Pi)$ and generated by
the forms $p_\pm$. The preceding considerations can be summarized as follows:
\begin{prop}
\label{P:estimatescounting}
For any  $R>0$ there exist $c>0$ and $M_R>0$
such that, with $\delta$ given by \eqref{eq-delta},
\[
\cN(P_+,-1-\lambda)-M_R\le \cN(Q,-1-\lambda)\le \cN(P_-,-1-\lambda)+M_R, \quad \lambda>0.
\]
\end{prop}

The  transversal part  (i.e. in variable $t$) in the above quadratic forms $p_\pm$, corresponds to some 1D Robin Laplacians with parameter $r + \kappa(s)/2$. 
In the both following sections, we will analyse their contributions by exploiting the following proposition proved in Appendix~\ref{aproof}. 

\begin{prop}\label{prop-1d}
Denote by $T^D$ (resp. $T^N$) the self-adjoint
operator in $L^2(0,\delta)$ acting as
$u\mapsto -u''$ with the boundary conditions $u(\delta)=0$ (resp. $u'(\delta)=0$)
and $u'(0)+ru(0)=0$, 
where $\delta$ is given by \eqref{eq-delta}. The first eigenvalues $E_{1}^{D/N}\equiv E_{1}$ and $E_{2}^{D/N}\equiv E_{2}$ satisfy in both cases, for $r\to+\infty$,
\begin{equation}
        \label{eq-evd}
E_1=-r^2+\cO(r^2 e^{-2r\delta}), \quad E_2\ge 0.
\end{equation}
Furthermore, the normalized and positive eigenfunctions $\psi^{D/N}\equiv \psi$ satisfy both
\begin{equation}
        \label{eq-efd}
\|\partial_r \psi\|^2_{L^2(0,\delta)}=\cO(r^{-2\rho}), \quad
\psi(0)^2=2r+\cO(r^2\delta e^{-2r\delta}), 
\end{equation}
and in the Neumann case
\[ 
\psi(\delta)^2=\cO(r^2\delta e^{-2r\delta}).
\]
\end{prop}

\section{Estimating from below the counting function}

We are going to obtain a lower bound for the counting function $\cN(P_+,-1-\lambda)$
by proving a majoration for the eigenvalues of $P_+$. 
The idea is to use quasimodes driven by the eigenfunction associated with $E_1(T^D)$ (see Proposition \ref{prop-1d}).

For shortness, denote by $\psi=\psi(r,t)$ and $E^{D}$ a normalized eigenfunction and the first eigenvalue of $T^{D}$. 
Denote
\begin{equation}
      \label{eq-u}
U:=(R,+\infty)\times \TT
\end{equation}
Let $u$ be a function of the form.
\[
u(r,s,t)=v(r,s)\psi(r,t), \quad v\in  H^1_0(U),
\]
Due to the normalization of $\psi$ we have,
\begin{equation}
\label{E:simplifnormpsi}
\|v_r\psi\|^2_{L^2(\Pi)}= \|v_r\|^2_{L^2(U)}, \quad
\langle v\psi_r,v_r\psi \rangle_{L^2(\Pi)}=0,
\end{equation}
and a direct computation shows that $p_+(u)= k(v)$, with
\begin{multline*}
k(v):=\int_U \bigg(
v_r^2 + \|\psi_r\|^2_{L^2(0,\delta)} v^2\\
+\dfrac{1}{r^2} \Big(
(1+C_{G}\delta) v_s^2 -\dfrac{\psi(r,0)^2}{2} \kappa v^2 +(E^{D}+C_{K})v^2
\Big)\bigg)\dd r \dd s.
\end{multline*}
In \eqref{eq-delta} set
$\rho:=3/4$,
then by Proposition~\ref{prop-1d}
one can find $C'>0$ such that for all $r>R$ there holds
\[
E^{D}+C_{K}\le -r^2+C' r^{\frac 12},\quad
\big|\psi(r,0)^2- 2r\big|\le C',
\quad
\|\psi_r\|^2_{L^2(0,\delta)}\le C'r^{-\frac 32},
\]
It follows that, with a suitable constant $a_{+}>0$, there holds
$k(v)\le -\|v\|^2+k_+(v)$,
\[
k_+(v):= \int_R^{+\infty} \int_\TT \bigg(
v_r(r,s)^2 + \dfrac{1+a_{+} r^{-\frac 34}}{r^2}\, v_s(r,s)^2
-\dfrac{\kappa(s)-a_{+} r^{-\frac12}}{r}\,v(r,s)^2\bigg)\dd s\dd r
\]
with
\[
\cD(k_+)=H^1_0\big((R,+\infty)\times \TT\big).
\]
It follows by the min-max principle that for any $j\in\NN$ there holds
\[
E_j(P_+)\le -1+E_j(K_+),
\]
where $K_+$ is the self-adjoint operator in $L^2(U)$
generated by the form $k_+$, and
\[
\cN(K_+,-\lambda)\le \cN(P_+,-1-\lambda), \quad \lambda>0,
\]
and, due to Proposition \ref{P:estimatescounting},
\[
\cN(Q,-1-\lambda) \ge \cN(K_+,-\lambda)-M_R, \quad \lambda>0.
\]
Therefore we get the lower bound of Theorem \ref{T:approx}.

\section{Estimating from above the counting function}
\label{S:lowerbound}
In this section, to study the quadratic form $p_-$ and the associated operator $P_-$, 
we will decompose the transversal part (i.e. the operator in variable $t$) into the space generated by the eigenvector associated to the first eigenvalue of $T^N$ (see Proposition \ref{prop-1d}) and its orthogonal space which will give a non-negative contribution.
First, we perform  additional simplifications to get rid of the last boundary term.
Let us recall that $\delta$ was given by \eqref{eq-delta}. For $b\in(0,c)$,
denote
\begin{gather*}
\Pi'_r:=\TT\times (0,b r^{-\rho}), 
\quad
\Pi''_r:=\TT\times (b r^{-\rho}, c r^{-\rho}), \quad
\quad r>R,\\
\Pi':=\big\{ (r,s,t): r\in (R,+\infty),\, (s,t)\in \Pi'_r\big\},\\
\Pi'':=\big\{ (r,s,t): r\in (R,+\infty),\, (s,t)\in \Pi''_r\big\},
\end{gather*}
and consider the  quadratic forms
\begin{multline}
           \label{eq-pm}
p'_-(u):=
\int_R^{+\infty}\Bigg( \int_{\Pi'_r}  \Big(\dfrac{\partial u}{\partial r}\Big)^2\, \dd s\,\dd  t\\
+\dfrac{1}{r^2} \bigg[
\int_{\Pi'_r} \bigg( (1-C_{G}\delta) \Big( \dfrac{\partial u}{\partial s}\Big)^2
+\Big( \dfrac{\partial u}{\partial t}\Big)^2 -C_{K} u^2 \bigg) \dd t\, \dd s\\
-\int_\TT \Big( r+ \dfrac{\kappa(s)}{2}\Big)u(r,s,0)^2\dd s 
 \bigg]\Bigg)\dd r, \quad \cD(p'_-)=H^1(\Pi'),
\end{multline}
and 
\begin{multline*}
p''_-(u):=
\int_R^{+\infty}\Bigg( \int_{\Pi''_r}  \Big(\dfrac{\partial u}{\partial r}\Big)^2\, \dd s\,\dd  t\\
+\dfrac{1}{r^2} \bigg[
\int_{\Pi''_r} \bigg( (1-C_{G}\delta) \Big( \dfrac{\partial u}{\partial s}\Big)^2
+\Big( \dfrac{\partial u}{\partial t}\Big)^2 -C_{K} u^2 \bigg) \dd t\, \dd s\\
-C_{B}\int_\TT u\big(r,s,\delta(r)\big)^2\dd s
 \bigg]\Bigg)\dd r, \quad \cD(p''_-)=H^1(\Pi'').
\end{multline*}
If $P'_-$ and $P''_-$ are the associated self-adjoint operators, acting respectively
in $L^2(\Pi')$ and $L^2(\Pi'')$, then we have the form inequality
\[
P_-\ge P'_-\oplus P''_-
\]
and, by the min-max principle
\[
E_j(P_-)\ge E_j(P'_-\oplus P''_-).
\]
We would like to show first that for a suitable choice of parameters one has
\begin{equation}
         \label{eq-ppp}
P''_-\ge -1.
\end{equation}
Due to the one-dimensional Sobolev inequality
\begin{equation}
        \label{eq-sobol}
f(0)^2\le x\|f'\|^2_{L^2(0,a)}+\dfrac{2}{x}\|f\|^2_{L^2(0,a)}, \quad
f\in H^1(0,a), \quad a>0, \quad x \in(0,a],
\end{equation}
it follows that in the expression for $p''_-$ we have
\begin{multline*}
\int_{\Pi''_r} \Big( \dfrac{\partial u}{\partial t}\Big)^2 \dd t\, \dd s
-C_{B}\int_\TT u\big(r,s,\delta(r)\big)^2\dd s\\
\ge
\big(1-(c-b)C_{B}r^{-\rho}\big)\int_{\Pi''_r} \Big( \dfrac{\partial u}{\partial t}\Big)^2 \dd t\, \dd s
-\dfrac{2 C_{B}}{c-b}\, r^{\rho}
\int_{\Pi''_r} u^2 \dd t\, \dd s.
\end{multline*}
In particular, the initial value $R>0$ can be chosen sufficiently large
to have
\[
(c-b)C_{B}R^{-\rho}<1 \text{ for } r>R,
\]
then
\[
p''_-(u)\ge
-\int_R^{+\infty} \int_{\Pi''_r}  \dfrac{1}{r^2}\Big(C_{K}+\dfrac{2 C_{B}}{c-b}\, r^{\rho}\Big) u^2 \dd t\, \dd s\, \dd r.
\]
Assume that $\rho\in(0,1)$, then we may increase the value of $R$ to have
\[
\dfrac{1}{r^2}\Big(C_{K}+\dfrac{2 C_{B}}{c-b}\, r^{\rho}\Big)\le 1 \text{ for } r>R,
\]
which gives \eqref{eq-ppp}, and it follows that
\[
\cN(P_-,-1-\lambda)=\cN(P'_-,-1-\lambda), \quad \lambda>0.
\]
In the rest of the section we change the definition of $\delta$
by setting
\begin{equation}
        \label{eq-delta2}
				\delta=b r^{-\rho}, \quad \rho\in(0,1) \Bk.
\end{equation}
Let $E$ and $\psi$ be the first eigenvalue and the associated normalized eigenfunction
of $T^N$ (see Proposition~\ref{prop-1d}). Furthermore, let $u\in\cD(p'_-)$ with $u\in C^1$.
 Represent
it in the form
\[
u(r,s,t)=v(r,s)\psi(r,t)+w(r,s,t), \quad v(r,s):=\int_0^\delta \psi(r,t)u(r,s,t)\dd t.
\]
Recall that $p'_{-}$ is defined in \eqref{eq-pm}. We aim at a lower bound for $p'_-(u)$, for $u$ of the above form. 
\begin{lemma}
For any $\varepsilon_{1}>0$ there exist $A(\varepsilon_{1})>0$, $C_K>0$, $C_G>0$
such that for all $u$ as above there holds 
\begin{multline}
           \label{eq-pm2}
p'_-(u)\ge \int_R^\infty \Big(1-\dfrac{A(\varepsilon_1)}{r^{2\rho}}\Big) v_r(r,s)^2\dd r \dd s
-\Big( \varepsilon_1+\dfrac{1}{R}\Big) \|w\|_{L^2(\Pi')}\\
+
\int_R^\infty \bigg(\dfrac{1}{r^2} 
\int_{\Pi'_r} \Big((1-C_{G}\delta) u_s^2
+u_t^2 -C_{K} u^2 - r^{-1} w_t^2\Big)\dd t\, \dd s\\
-\int_\TT  (r+ \dfrac{\kappa(s)}{2})u(r,s,0)^2\dd s 
\bigg)
 \dd r.
\end{multline}
\end{lemma}
\begin{proof}
We have
$\|u_r\|^2_{L^2(\Pi')} \geq \|v_r\psi\|^2_{L^2(\Pi')}+2\langle v \psi_{r}+w_{r},v_r\psi \rangle_{L^2(\Pi')}$.
Using identities as in \eqref{E:simplifnormpsi}, we arrive at
\begin{equation}
   \label{eq-a0}
\|u_r\|^2_{L^2(\Pi')}
\ge \|v_r\|^2_{L^2(U)} + 2\int_U v_r(r,s) \int_0^\delta \psi(r,t)w_r(r,s,t) \dd t \, \dd r \dd s.
\end{equation}
By construction,
\[
\int_0^\delta \psi(r,t)w(r,s,t) \dd t=0,
\]
hence, by taking the derivative with respect to $r$ we obtain
\[
\int_0^\delta \psi_r(r,t)w(r,s,t) \dd t+\int_0^\delta \psi(r,t)w_r(r,s,t) \dd t
+\psi(r,\delta)w(r,s,\delta) \delta'(r)=0, 
\]
and
\begin{multline}
      \label{eq-a1}
2\int_U v_r(r,s) \int_0^\delta \psi(r,t)w_r(r,s,t) \dd t \, \dd r \dd s\\
=-2\langle v_r\psi_r, w\rangle_{L^2(\Pi')}-2 \int_U v_r(r,s) \psi(r,\delta)w(r,s,\delta) \delta'(r)\dd r \dd s.
\end{multline}
Using Proposition~\ref{prop-1d} we estimate the first term, with $\varepsilon_1>0$:
\begin{multline}
  \label{eq-a2}
-2\langle v_r\psi_r, w\rangle_{L^2(\Pi')}\ge -\varepsilon_1 \|w\|^2_{L^2(\Pi')}-\dfrac{1}{\varepsilon_1}\|v_r\psi_r\|^2_{L^2(\Pi')}\\
=-\varepsilon_1 \|w\|^2_{L^2(\Pi')}-\dfrac{1}{\varepsilon_1}\int_U \|\psi_r\|_{L^2(0,\delta)}v_r(r,s)^2\dd r \dd s\\
\ge -\varepsilon_1 \|w\|^2_{L^2(\Pi')}-\dfrac{a}{\varepsilon_1}\int_U \dfrac{v_r(r,s)^2}{r^{2\rho}} \dd r \dd s.
\end{multline}
The second term in \eqref{eq-a1} is estimated using a combination of the Cauchy-Schwarz inequality
and the Sobolev inequality \eqref{eq-sobol}:
\begin{gather*}
		-2 \int_U v_r(r,s) \psi(r,\delta)w(r,s,\delta) \delta'(r)\dd r \dd s\\
		\ge
		-\int_U \delta'(r)^2 v_r(r,s)^2\dd r \dd s
		-\int_U \psi(r,\delta)^2w(r,s,\delta)^2\dd r \dd s\\
		\ge -\int_U \delta'(r)^2 v_r(r,s)^2\dd r \dd s
		-\int_U \delta \psi(r,\delta)^2 \int_0^\delta w_t(r,s,t)^2\dd t\dd r \dd s\\
		-2\int_U \delta^{-1} \psi(r,\delta)^2 \int_0^\delta w(r,s,t)^2\dd t\dd r \dd s
\end{gather*}
Using the estimate of Proposition~\ref{prop-1d} for $\psi(r,\delta)^2$
and the explicit form of $\delta$ we conclude that, for $r>R$ and $R$ chosen sufficiently large,
\begin{multline}
    \label{eq-a4}
		-2 \int_U v_r(r,s) \psi(r,\delta)w(r,s,\delta) \delta'(r)\dd r \dd s\\
				\ge -b^2\rho^2\int_U r^{-2-2\rho} v_r(r,s)^2\dd r \dd s
		-\int_U r^{-3} \int_0^\delta w_t(r,s,t)^2\dd t\dd r \dd s\\
		-2\int_U r^{-1} \int_0^\delta w(r,s,t)^2\dd t\dd r \dd s
\end{multline}
The substitution of \eqref{eq-a2} and \eqref{eq-a4} into \eqref{eq-a1} and then into
\eqref{eq-a0} gives
\begin{multline*}
\|u_r\|^2_{L^2(\Pi')} \ge \int _0 \Big(1-\dfrac{A(\varepsilon_1)}{r^{2\rho}}\Big) v_r(r,s)^2\dd r \dd s\\
-\Big( \varepsilon_1+\dfrac{1}{R}\Big) \|w\|_{L^2(\Pi')}-
\int_{\Pi'}\dfrac{w_t^2}{r^3}\,\dd r \dd s \dd t.
\end{multline*}
Substituting into \eqref{eq-pm} we get the lemma.
\end{proof}
Now we show a lower bound for the second term of \eqref{eq-pm2}, i.e. for
\begin{multline*}
I(r):=\int_{\Pi'_r} \bigg( (1-C_{G}\delta) u_s^2
+u_t^2 -C_{K} u^2 - r^{-1} w_t^2\bigg) \dd t\, \dd s\\
\quad-\int_\TT \Big( r+ \dfrac{\kappa(s)}{2}\Big)u(r,s,0)^2\dd s.
\end{multline*}
\begin{lemma}
There exists $a_{-}>0$ and $R>0$ such that for all $r>R$ and all $u$
there holds 
\begin{multline*}
I(r)\ge (1-C_{G}r^{-\rho}) \|v_s\|^2_{L^2(\TT)} - r\int_\RR \kappa(s) v(r,s)^2\dd s\\
- \dfrac{r^2}{2} \|w\|^2_{L^2(\Pi'_r)}
-(r^2+a_{-})\|v\|^2_{L^2(\TT)}.
\end{multline*}
\end{lemma} 
 \begin{proof}
For shortness, denote $E^{N}:=E_1(T^N)$. Remark first that due to the choice of $\psi$ there holds
\begin{gather*}
\|\psi_t v\|^2_{L^2(\Pi'_r)}-r \psi(r,0)^2\int_\TT v(r,s)^2\dd s=E^{N} \|\psi v\|^2_{L^2(\Pi'_r)}\equiv E^{N}\|v(r,\cdot)\|^2_{L^2(\TT)},\\
\big\langle \psi_t v, w_t\rangle_{L^2(\Pi'_r)}
-r \psi(r,0)\int_\TT v(r,s) w(r,s,0)\dd s=0,\\
\big\langle \psi v_s, w_s\rangle_{L^2(\Pi'_r)}=0,\\
\end{gather*}
and
\begin{gather*}
I(r)=(1-C_{G}\delta) \|v_s\|^2_{L^2(\TT)} - \int_\RR \dfrac{\kappa(s)}{2} v(r,s)^2\psi(r,0)^2\dd s\\
+ (1-C_{G}\delta) \|w_s\|^2_{L^2(\Pi'_r)}+(1-r^{-1}) \|w_t\|^2_{L^2(\Pi'_r)}-C_{K} \|w\|^2_{L^2(\Pi'_r)}\\
-\int_\TT \kappa(s) v(r,s)\psi(r,0)w(r,s,0)\dd s\\
-\int_\TT \Big( r+ \dfrac{\kappa(s)}{2}\Big)w(r,s,0)^2\dd s
+(E^{N}-C_{K})\|v\|^2_{L^2(\TT)}.
\end{gather*}
We estimate, with any $m>0$,
\begin{gather*}
\Big|
\int_\TT \kappa(s) v(r,s)\psi(r,0)w(r,s,0)\dd s
\Big| \le \dfrac{m}{r} \psi(r,0)^2 \|v\|^2_{L^2(\TT)} + \dfrac{r}{m}\|\kappa w(r,\cdot,0)\|^2_{L^2(\TT)}.
\end{gather*}
Therefore, for $B:=\|\kappa^2\|_\infty + \|\kappa\|_\infty$ we have
\begin{gather*}
I(r)\ge (1-C_{G}\delta) \|v_s\|^2_{L^2(\TT)} - \psi(r,0)^2\int_\RR \Big(\dfrac{\kappa(s)}{2}+\dfrac{m}{r}\Big) v(r,s)^2\dd s\\
+ (1-C_{G}\delta) \|w_s\|^2_{L^2(\Pi'_r)}+(1-r^{-1}) \|w_t\|^2_{L^2(\Pi'_r)}\\
- r\Big(1+\dfrac B m\Big)\int_\TT w(r,s,0)^2\dd s
+(E^{N}-C_{K})\|v\|^2_{L^2(\TT)}.
\end{gather*}
By construction we have
\[
\|w_t\|^2_{L^2(\Pi'_r)}-r \int_\TT w(r,s,0)^2\dd s\ge 0,
\]
and for any $\mu\in(0,1)$ we have
\[
\|w_t\|^2_{L^2(\Pi'_r)}\ge \mu \|w_t\|^2_{L^2(\Pi'_r)} + (1-\mu)r \int_\TT w(r,s,0)^2\dd s,
\]
from which we infer
\begin{gather*}
I(r)\ge (1-C_{G}\delta) \|v_s\|^2_{L^2(\TT)} - \psi(r,0)^2\int_\RR \Big(\dfrac{\kappa(s)}{2}+\dfrac{m}{r}\Big) v(r,s)^2\dd s\\
{}+ (1-C_{G}\delta) \|w_s\|^2_{L^2(\Pi'_r)}+(\mu-r^{-1}) \|w_t\|^2_{L^2(\Pi'_r)}-C_{K} \|w\|^2_{L^2(\Pi'_r)}\\
{}- r\Big(\mu+\dfrac B m\Big)\int_\TT w(r,s,0)^2\dd s
+(E^{N}-C_{K})\|v\|^2_{L^2(\TT)}.
\end{gather*}
Using \eqref{eq-sobol} we obtain, with $\sigma>0$ as $r$ is large,
\[
\int_\TT w(r,s,0)^2\dd s\le \dfrac{\sigma}{r} \|w_t\|^2_{L^2(\Pi'_r)} + \dfrac{2r}{\sigma}\|w\|^2_{L^2(\Pi'_r)},
\]
and
\begin{gather*}
I(r)\ge (1-C\delta_{G}) \|v_s\|^2_{L^2(\TT)} - \psi(r,0)^2\int_\RR \Big(\dfrac{\kappa(s)}{2}+\dfrac{m}{r}\Big) v(r,s)^2\dd s\\
+ (1-C\delta_{G}) \|w_s\|^2_{L^2(\Pi'_r)}+\Big( (1-\sigma)\mu-\dfrac{1}{r}-\dfrac{\sigma B}{ m}\Big) \|w_t\|^2_{L^2(\Pi'_r)}\\
- \dfrac{2r^2 \Big(\mu+\frac B m\Big)+\sigma C_{K}}{\sigma} \|w\|^2_{L^2(\Pi'_r)}
+(E^{N}-C_{K})\|v\|^2_{L^2(\TT)}.
\end{gather*}
Now take $\sigma=1/2$ and assume that $m$ and $R$ are sufficiently large so that, for all $r>R$,
we can choose
\[
\mu=\dfrac{2}{r}+\dfrac{B}{m}\in (0,1).
\]
We this choice of $\mu$, the coefficient in front of $\|w_t\|^2_{L^2(\Pi'_r)}$ vanishes, and we get
\begin{multline*}
I(r)\ge (1-C_{G}\delta) \|v_s\|^2_{L^2(\TT)}\\
 - \psi(r,0)^2\int_\RR \Big(\dfrac{\kappa(s)}{2}+\dfrac{m}{r}\Big) v(r,s)^2\dd s
+ (1-C_{G}\delta) \|w_s\|^2_{L^2(\Pi'_r)}\\
- \left(8r^2 \Big(\dfrac{1}{r}+\dfrac B m\Big)+C_{K}\right) \|w\|^2_{L^2(\Pi'_r)}
+(E^{N}-C_{K})\|v\|^2_{L^2(\TT)},
\end{multline*}
and choosing $R$ and $m$ sufficiently large and estimating $\|w_s\|\ge 0$ we have
\begin{multline*}
I(r)\ge (1-C_{G}\delta) \|v_s\|^2_{L^2(\TT)} - \psi(r,0)^2\int_\RR \Big(\dfrac{\kappa(s)}{2}+\dfrac{m}{r}\Big) v(r,s)^2\dd s\\
- \dfrac{r^2}{2} \|w\|^2_{L^2(\Pi'_r)}
+(E^{N}-C_{K})\|v\|^2_{L^2(\TT)}, \quad r>R.
\end{multline*}
Using the estimates of Proposition \ref{prop-1d} we have the lemma.
\end{proof}
The substitution in \eqref{eq-pm2} gives
\begin{multline*}
p'_-(u)\ge -\|v\|^2_{L^2(U)}- \Big( \dfrac{1}{2}+\varepsilon_1+\dfrac{1}{R}\Big)\|w\|_{L^2(\Pi')}\\
+\int_R^\infty \int_\TT \bigg[\Big(1-\dfrac{A(\varepsilon_1)}{r^{2\rho}}\Big) v_r(r,s)^2
+\dfrac{1-C_{G}r^{-\rho}}{r^2} v_s(r,s)^2{}\\
{}-\dfrac{\kappa(s)+a_- r^{-1}}{r}\,v(r,s)^2
 \bigg]\dd s\dd r,
\end{multline*}
 which holds by density for all $u\in \cD(p'_-)$. Choose $\varepsilon_1>0$ sufficiently small and $R$ sufficiently large to have
\[
 \dfrac{1}{2}+\varepsilon_1+\dfrac{1}{R} < 1,
\]
then fix the corresponding constant $A=A(\varepsilon_1)$ and set $\rho=3/4$.
Let $K_-$ be the self-adjoint operator acting in $L^2(U)$ and generated by the quadratic form
\begin{multline*}
k_-(v)=\int_{U} \Big((1-A r^{-\frac 32}) v_r(r,s)^2\\
+\frac{1-C_{G}r^{-\frac{3}{4}}}{r^2} v_s(r,s)^2{}-\frac{\kappa(s)+a_-r^{-1}}{r}\,v(r,s)^2
\Big) \dd s\dd r,
\end{multline*}
\Bk
defined on $\cD(k_-)=H^1(U)$, with $U$ given by \eqref{eq-u}.
Then
\[
E_j(P'_-)\ge -1+E_j(K_-), \quad \text{ if } E_j(K_-)<0,
\]
hence,
\[
\cN(P'_-,-1-\lambda)\le \cN(K_-,-\lambda), \quad \lambda>0,
\]
and Proposition \ref{P:estimatescounting} provides
\[
\cN(Q,-1-\lambda)\le \cN(K_-,-\lambda)+M_R, \quad \lambda>0,
\]
which is the upper bound of Theorem \ref{T:approx}.

\section{Weyl asymptotics for the model operator}
\label{S:Weylmodel}
In this section, for $x\in\RR$ we denote $x_+:=\max\{x,0\}$,
and for a function $f:A\to\RR$,  we denote $f_+(x):=\max\big\{f(x),0\big\}$.

 The final result will be a consequence of the following
proposition: 
\begin{prop}\label{PropModel}
Let $\ell>0$ and $V:[0,\ell]\to \RR$ and
$a,b,\nu : \RR_+ \to \RR $ be continuous functions 
with
\begin{gather}
\inf_{r \in \RR_+} a(r) >0, \quad \inf_{r \in \RR_+} a(r) >0 ,\nonumber
\\
\label{E:approxfct}
\lim_{r \rightarrow + \infty} a(r)=\lim_{r \rightarrow + \infty} b(r)=1,
\quad  \lim_{r \rightarrow + \infty} \nu(r)=0.
\end{gather}
\Bk
Consider the quadratic form
\[
k(v)=\int_R^\infty \int_0^\ell \Big(a(r)v_r^2
+\dfrac{b(r)}{r^2}\, v_s^2 -\dfrac{V(s)+\nu(r)}{r}\, v^2 \Big)\, \dd s\, \dd r
\]
defined on $H^1(\Pi)$ or $ H^1_0(\Pi)$, with
$\Pi=(R,+\infty)\times(0,\ell)$, $R\geq 0$,
and let $K$ be the associated self-adjoint operator acting in $L^2(\Pi)$.
Then
\[
\cN(K,-\lambda)=\dfrac{1}{4\pi \lambda} \int_0^\infty \int_0^\ell \big( V(s)-x\big)_+\dd s \dd x
+o\Big(\dfrac{1}{\lambda}\Big), \quad  \lambda \to 0+.
\]
\[= \dfrac{1}{8\pi\lambda }  \int_0^\ell  V_+(s)^2 \dd s
+o\Big(\dfrac{1}{\lambda}\Big), \quad  \lambda \to 0+.
\]
\end{prop}

\begin{proof}
Let us introduce  the new variable $x=\lambda (r-R)$ and for a domain $\Omega\subset \Pi_0=(0,+\infty)\times(0,\ell)$, we consider the quadratic form
 \[
l_{\Omega,\lambda}(w)=\int_\Omega \Big(a_\lambda(x)\lambda w_x^2
+\dfrac{b_\lambda(x)}{(x+\lambda R)^2}\,\lambda w_s^2 -\dfrac{V(s)+\nu_\lambda(x)}{x+\lambda R}\, w^2 \Big)\, \dd s\, \dd x,
\]
where for a function $g$ defined on $\RR^+$ we set $g_\lambda(x):= g(R + x/\lambda)$. We denote by $L_\Omega^{D/N}(\lambda)$ the associated self-adjoint operator acting in $L^2(\Omega)$ with the Dirichlet (or Neumann) boundary condition on $\partial \Omega$ and by ${\mathfrak N}^{D/N}( \Omega, \lambda)$, the associated counting function:
\begin{equation}\label{defNgot}
{\mathfrak N}^{D/N}( \Omega, \lambda):= \cN(L_\Omega^{D/N}(\lambda),-1) .
\end{equation}

We clearly have 
 \[ \cN(K,-\lambda)= {\mathfrak N}^{D/N}( \Pi_0, \lambda).\]
Now, fix $M>0$ such that $M > V(s)+ \nu(r)$ for any $(s,r)\in \Pi$. Then exploiting that, on $[M , + \infty) \times (0,\ell)$, we have
\[-\dfrac{V(s)+\nu_\lambda(x)}{x+\lambda R} >-1, \]
we deduce the following Dirichlet-Neumann bracketing for the counting functions:
\[  
{\mathfrak N}^{D}\big( (0,M) \times (0,\ell), \lambda\big) \leq {\mathfrak N}^{D/N}( \Pi_0, \lambda) \leq {\mathfrak N}^{N}\big( (0,M) \times (0,\ell), \lambda\big).
\]

Then we deduce Proposition \ref{PropModel} from the following Lemma.
\end{proof}

\begin{lemma}\label{lemCut}
For $M > \sup_\Pi \big(V(s)+ \nu(r)\big)$, we have, as $ \lambda\to 0+$:
\begin{gather*}
{\mathfrak N}^{D/N}\big( (0,M) \times (0,\ell), \lambda\big)
=\dfrac{1}{4\pi \lambda} \int_0^M \int_0^{\ell} \big( V(s)-x \big)_+\dd s \dd x+o\Big(\dfrac{1}{\lambda}\Big).
\end{gather*}

\end{lemma}

\begin{proof}
We adapt a textbook proof for Schr\"odinger operators, see e.g.  \cite[Theorem XIII.78]{ReSi78}.
For $(m,n) \in \NN$ fixed, consider the rectangles
\[ \cC(j,k):= \left( \dfrac{(j-1)M}{m},\dfrac{jM}{m} \right) \times \left( \dfrac{(k-1)\ell}{n},\dfrac{k\ell}{n} \right),\ 
j= 2, \cdots , m, \ k=1, \cdots , n,
\]
\[ \cR :=  \left(0,\frac{M}{m}\right) \times (0,\ell),
 \]
then
\begin{multline}\label{eqmodel1}
\sum_{j,k} {\mathfrak N}^{D}\big(\cC(j,k), \lambda\big) \leq {\mathfrak N}^{D/N}\big( (0,M) \times (0,\ell), \lambda\big)\\ \leq 
\sum_{j,k} {\mathfrak N}^{N}\big(\cC(j,k), \lambda\big) + {\mathfrak N}^{N}\big(\cR, \lambda\big).
\end{multline}
In what follows, we will estimate separately the contributions coming from $\cC(j,k)$ and $\cR$.
\\
\\
{\sc Step 1: estimates on $\cC(j,k)$.} Denote
\[
x^-(j):=\dfrac{(j-1)M}{m}, \quad x^+(j):=\dfrac{jM}{m}.
\]
For a continuous function $g$, introduce
\begin{align*}
g^{+}(j)&:= \max \big\{ g(x); \; x \in [x^-(j),x^+(j)]\big\}, \\
g^{-}(j)&:= \min \big\{ g(x); \; x \in [x^-(j),x^+(j)]\big\},
\end{align*}
Similarly, denote
\begin{align*}
V_+(k)&:=\max\Big\{ V(s): s\in \Big[\frac{(k-1)\ell}{n},\frac{k\ell}{n}\Big] \Big\},\\
V_-(k)&:=\min\Big\{ V(s): s\in \Big[\frac{(k-1)\ell}{n},\frac{k\ell}{n}\Big]\Big\}
\end{align*}
Define the quadratic forms with frozen coefficients:
\begin{multline*}
l^{\pm}_{\cC(j,k),\lambda}(w)=\int_{\cC(j,k)} \Big( a^\pm_\lambda(j)\lambda w_x^2
\\
+\dfrac{b^\pm_\lambda(j)}{(x^\mp(j)+\lambda R)^2}\,\lambda w_s^2
-\dfrac{V^\mp(k)+\nu^\mp_\lambda(j)}{x^{\pm \epsilon}(j)+\lambda R}\, w^2 \Big)\, \dd s\, \dd x,
\end{multline*}
on the same domain as $l_{\cC(j,k),\lambda}$, with $\epsilon =$ sign $(V^\mp(k)+\nu^\mp_\lambda(j))$. Obviously, 
\begin{equation}\label{eqmodel2}
   l^-_{\cC(j,k),\lambda} \leq l_{\cC(j,k),\lambda} \leq l^+_{\cC(j,k),\lambda},
 \end{equation}
Let $L_{\cC(j,k)}^{+}(\lambda)$ and $L_{\cC(j,k)}^{-}(\lambda)$ be the self-adjoint operators associated with the quadratic forms $l^{+}_{\cC(j,k),\lambda}$ and $l^{-}_{\cC(j,k),\lambda}$, with Dirichlet and Neumann boundary condition respectively and denote 
\[ {\mathfrak N}^{\pm}\big(\cC(j,k), \lambda\big) := \cN(L_{\cC(j,k)}^{\pm}(\lambda),-1),\]
then
\begin{equation}\label{eqmodel2b}
{\mathfrak N}^{+}\big(\cC(j,k), \lambda\big) \leq {\mathfrak N}^{D}\big(\cC(j,k), \lambda\big) \leq  {\mathfrak N}^{N}\big(\cC(j,k), \lambda\big) \leq {\mathfrak N}^{-}\big(\cC(j,k), \lambda\big).
\end{equation}
The eigenvalues of $L_{\cC(j,k)}^{\pm}(\lambda)$  are explicit:
\[
\Lambda_{\kappa,\tau}^\pm (j,k,\lambda) =
a^\pm_\lambda(j)\lambda \dfrac{\pi^2 m^2}{M^2} \, \kappa^2 
+\dfrac{b^\pm_\lambda(j)}{(x^\mp(j)+\lambda R)^2}\,\lambda \, 
\dfrac{\pi^2 n^2}{\ell^2} \, \tau^2  -\dfrac{V^\mp(k)+\nu^\mp_\lambda(j)}{x^{\pm \epsilon} (j)+\lambda R}
\]
with $(\kappa,\tau) \in N^2$ with $N=\NN$ or $N=\NN^*$ depending on the boundary condition. For $A>0$, $B>0$, $C \in \RR$ as $\lambda \to 0+$ we have
\begin{multline}
 \label{eq-ell}
\# \Big\{ (\kappa,\tau) \in N^2:\  \dfrac{\kappa^2}{A^2}+\dfrac{\tau^2}{B^2}\leq \frac{C}{\lambda}\Big\}\\ = \text{area} \Big\{ (X,Y) \in \RR_+^2; \; \dfrac{X^2}{A^2}+\dfrac{Y^2}{B^2}\leq \frac{C}{\lambda}\Big\} + o\Big(\dfrac{1}{\lambda}\Big)
 = \frac{\pi A BC_+}{4\lambda} + o\Big(\dfrac{1}{\lambda}\Big),
\end{multline}
and exploiting that, under the assumptions \eqref{E:approxfct}, for $ j \in \NN^*$, as $\lambda \to 0+$, we have,
uniformly in $j$,
\[ a^\pm_\lambda(j) \to 1; \quad b^\pm_\lambda(j) \to 1; \quad  \nu^\pm_\lambda(j) \to 0; \quad \epsilon \to   \hbox{sign} V^\mp(k),
\]
we obtain
\begin{gather*}
\# \Big\{ (\kappa,\tau) \in N\times N: \; \Lambda_{\kappa,\tau}^\pm (j,k,\lambda) \leq -1 \Big\}\\
= \dfrac{\big(V^\mp(k)-x^{\pm \epsilon}(j)\big)_+}{4\pi \lambda} \times \dfrac{M\ell x^\mp(j)}{mnx^{\pm \epsilon}(j)} + o\Big(\dfrac{1}{\lambda}\Big),
\end{gather*}
where $o(1/\lambda)$ is uniform with respect to $(j,k)$ and depends on $(m,n)$ only. 
Note that when $V^\mp(k) < 0$, $\epsilon = -1$ and $\big(V^\mp(k)-x^{\pm \epsilon}(j)\big)_+=0$. Then the main contribution is non zero only for $\epsilon = +1$ and 
 we deduce the following asymptotics of the counting functions in the (small) squares $\cC(j,k)$, $j\geq 2$, $k\geq 1$:
\begin{equation}\label{eqmodel3}
{\mathfrak N}^{\pm}(\cC(j,k), \lambda) =  \dfrac{W^\pm_{m,n}(j,k)}{4\pi \lambda}  + o\Big(\dfrac{1}{\lambda}\Big)
\end{equation}
uniformly with respect to $(j,k)$ with
\[
W^\pm_{m,n}(j,k):=\dfrac{M\ell}{mn} \times (V^\mp(k)-x^\pm(j))_+ \times \dfrac{x^\mp(j)}{x^\pm(j)}.
\]
By definition of $x^\pm(j)$, we have 
\[ \dfrac{x^-(j)}{x^+(j)} = 1 -  \dfrac{1}{j} , \qquad \dfrac{x^+(j)}{x^-(j)} =  1 +  \dfrac{1}{j-1}. \]
Then using the convergence of Riemann sums to integrals and the estimate 
\[
\lim_{m \to \infty} \dfrac{1}{m} \sum_{j=1}^m \dfrac{1}{j} =0
 \]
we have 
\begin{equation}\label{eqmodel4}
 \sum_{j=2}^m \sum_{k=1}^n W^\pm_{m,n}(j,k)=
 \int_0^M \int_0^{\ell} \big( V(s)-x \big)_+\dd s \dd x  \, + \, \varepsilon_1(m,n), 
 \end{equation}
with $ \lim_{(m,n) \to \infty}  \varepsilon_1(m,n) =0$.
We deduce from \eqref{eqmodel2b},  \eqref{eqmodel3} and \eqref{eqmodel4} that for any $(m,n)\in \NN^2$ one has 
\begin{multline}\label{step1}
\sum_{j,k} {\mathfrak N}^{D/N}\big(\cC(j,k), \lambda\big)\\
=\dfrac{1}{4\pi \lambda} \int_0^M \int_0^{\ell} \big( V(s)-x \big)_+\dd s \dd x \, + \, \dfrac{\varepsilon_1(m,n)}{\lambda} + \dfrac{{\varepsilon}_{2,m,n}( \lambda)}{ \lambda},
\end{multline}
with $ \lim_{(m,n) \to \infty}  \varepsilon_1(m,n) =0$ and for any fixed $m$, $n$, $\lim_{\lambda  \to 0}  {\varepsilon}_{2,m,n}( \lambda) =0$.
\\
\\
{\sc Step 2: estimate on $\cR$.} 
Here our goal is to prove that the contribution on $\cR$ is negligible.
For suitable $\sigma>0$ and $\rho>0$,
one estimates in $\cR$, uniformly in $m$ and $\lambda$,
\[
a_\lambda^-(1)\ge \sigma,
\quad
b_\lambda^-(1)\ge \sigma,
\quad
\max V+\nu^+_\lambda(1)\le \rho.
\Bk
\]
With the above notation  we have 
\[
l_{\cR ,\lambda}(w) \geq  \int_\cR \sigma\lambda w_x^2
+\dfrac{1}{x+\lambda R} \Big( \dfrac{ \sigma m \lambda}{M + \lambda m R}  \, w_s^2 - \rho\, w^2 \Big) \, \dd s\, \dd x.
\]
The eigenvalues of the self-adjoint operator associated with the quadratic form
\[
w\mapsto \int_0^\ell \Big(\dfrac{ \sigma m \lambda}{M + \lambda m R}  \, w_s^2 - \rho\, w^2 \Big)\, \dd s,
\quad w\in H^1(0,\ell),
\]
 are 
\[
\mu_j(\lambda)=  \dfrac{ \sigma\lambda m}{M + \lambda m R} \dfrac{ \pi^2 \, j^2}{\ell^2} -\rho, \quad j=0,1,2,\dots,
\]
and for a suitable $c>0$ we have, uniformly in $m\ge 1$,
\begin{equation} \label{nbj}
\# \{ j:\  \mu_j(\lambda) <0 \} \le \dfrac{c}{\sqrt{m\lambda}}.
\end{equation}
Let $L_j(\lambda)$ be the self-adjoint operator associated with the quadratic form
\[
l_{j}(w)=\int_0^{\frac{M}{m}} \Big(  \sigma\lambda w_x^2
+\dfrac{\mu_j(\lambda)}{x+\lambda R}\,  \, w^2 \Big)\, \dd x,
\quad w\in H^1\Big( 0, \dfrac{M}{m}\Big),
\]
then 
\begin{equation}
\label{eq-nnn}
{\mathfrak N}\big( \cR, \lambda\big)  \leq \sum_{j\geq0} \cN\big(L_j(\lambda),-1\big).
\end{equation}
Due to $\mu_j(\lambda) \geq \mu_0(\lambda)=-\rho$ 
we have $\cN(L_j(\lambda),-1) \leq \cN(L_0(\lambda),-1)$,
while $L_0(\lambda)$ acts in $L^2(0,M/m)$ by
\[
w\mapsto -\sigma\lambda w''-\dfrac{\rho}{x+\lambda R} \, w
\]
with Neumann boundary condition. 

%
%

Since $x \mapsto 1/\sqrt{x}$ is integrable on $(0,M/m)$, the standard one-dimensional Weyl asymptotics yields 
\[
\cN(L_0(\lambda),-1)\le
 \dfrac{c_1}{\sqrt \lambda}
 \]
for some $c_1>0$.
By combining the last inequality with \eqref{nbj} and \eqref{eq-nnn}
we arrive at 
\begin{equation}\label{step2}
 {\mathfrak N}^{N}\big( \cR, \lambda\big) \leq \dfrac{C}{\sqrt{m} \, \lambda} 
\end{equation}
with some $C>0$ independent of $m$ and $\lambda$.

{\sc Conclusion.}
As follows from \eqref{eqmodel1}, \eqref{step1} and \eqref{step2}, for any positive integers $m$ and $n$
we have
\[{\mathfrak N}^{D/N}\big( (0,M) \times (0,\ell), \lambda\big)
=\dfrac{1}{4\pi \lambda} \int_0^M \int_0^{\ell} \big( V(s)-x \big)_+\dd s \dd x \, + \, \dfrac{\tilde{\varepsilon}_1(m,n)}{\lambda} + \dfrac{{\varepsilon}_{2,m,n}( \lambda)}{ \lambda},
\]
where $ \lim_{(m,n) \to \infty}   \tilde{\varepsilon}_1(m,n) =0$,
and for any fixed $m$ and $n$ one has $\lim_{\lambda  \to 0}  {\varepsilon}_{2,m,n}( \lambda) =0$.

We conclude Lemma \ref{lemCut} by choosing  $m_0$ and $n_0$ sufficiently large such that $\tilde{\varepsilon}_1(m_0,n_0)$
is sufficiently small and then $\lambda$ sufficiently small depending on $(m_0,n_0)$ fixed.
\end{proof}

\begin{rem}
Let us mention the semi-classical aspect of Proposition \ref{PropModel}. Actually, up to multiplication by $1/(4 \pi^2)$, the main contribution is the volume of the region of the phase space:
\[ \Big\{ (r,s,\xi, \eta) \in ([R,+ \infty) \times (0,\ell) \times \RR^2; \; \xi^2 + \dfrac{\eta^2}{r^2}- \dfrac{V(s)}{r} \leq - \lambda\Big\}.\]
which is given by:
\[ \pi \int_R^{+ \infty} \int_0^{\ell} \Big(\dfrac{V(s)}{r}  - \lambda \Big)_+ \, r \dd s \dd r = \dfrac{\pi}{ \lambda} \int_0^{+ \infty}  \int_0^{\ell} \big( V(s)-x \big)_+\dd s \dd x + o\Big(\dfrac{1}{\lambda}\Big). \qed
\]
\end{rem}

Now we complete the proof of Theorem~\ref{thm3}. Consider first the case $\alpha=1$, then, by Theorem \ref{thm2}, the decomposition
of the boundary of $\partial\Theta$ into $m$ closed loops $\gamma_j$ of length $\ell_j$
produces a decomposition
\[
K^\pm= \oplus_{j=1}^m K^\pm_j,
\]
where each $K^\pm_j$ corresponds to a maximal connected components of $\partial\Theta$
and is associated with the quadratic form
\[
k^{\pm}_j(v)=\int_R^\infty \int_{\gamma_j} \Big(a_{\pm}(r)v_r^2
+\dfrac{b_{\pm}(r)}{r^2}\, v_s^2 -\dfrac{\kappa(s)+\nu(r)}{r}\, v^2 \Big)\, \dd s\, \dd r
\]
defined on $\cD(k^+_j)=H^1_0(U_j)$ or $\cD(k^-_j)=H^1(U_j)$, $U_j:=(R,+\infty)\times\gamma_j$.
The natural identification of $\gamma_j$ with $\RR/(\ell_j \ZZ)$ gives then the form inequalities
\[
k^N_j\le k^-_j, \quad k^{+}_j\le k^D_j,
\]
with $k^N_j$ given by the same expression as $ k^-_j$ but defined
on the larger domain $\cD(k^N_j)=H^1\big((R,+\infty)\times(0,\ell_j)\big)$
and $k^D_j$ is the restriction of $k^{+}_j$ to $H^1_0\big((R,+\infty)\times(0,\ell_j)\big)$.
The operators associated with $k^{D/N}_j$ are now covered by Proposition~\ref{PropModel}, which gives
the proof for $\alpha=1$. For general $\alpha>0$, it is sufficient to use the identity \eqref{eq-ej}.

\appendix

\section{Proof of Proposition~\ref{prop-1d}}\label{aproof}

In this section, we study two 1D problems on $\big(0,\delta(r)\big)$ with Robin condition (with parameter $r$) at $0$ and with Dirichlet (resp. Neumann) condition at $\delta(r)$ in order to prove Proposition \ref{prop-1d}. The only novelty lies on the $L^2$-estimate of the $r$-derivative  of the ground state, but for the sake of completeness we provide the proofs of all the statements.

Let us look for eigenvalues of the form $E^{D/N}=-(k^{D/N}r)^2$, $k^{D/N}>0$, then
the boundary condition $u(\delta)=0$ (respectively $u'(\delta)=0$) give the following forms for the positive normalized eigenfunctions:
\begin{equation}
       \label{eq-forpsiDN}
\psi^{D}(t)=C^{D}(r)\sinh\big(k^{D}r(\delta-t)\big) \ \ \mbox{and} \ \ \psi^{N}(t)=C^{N}(r)\cosh\big(k^{N}r(\delta-t)\big),
\end{equation}
where $C^{D/N}(r)>0$ are normalization factors. The second boundary condition
gives then
\begin{equation}
       \label{eq-kkd}
k^{D}\coth (k^{D}r\delta)=1 \ \ \mbox{and} \ \ k^{N}\tanh (k^{N}r\delta)=1 
\end{equation} which can be rewritten as
$F^{D/N}(kr\delta)=r\delta$ with $F^{D}(t)=t\coth t$ and $F^{N}(t)=t\tanh t$.
The function $F^{D}$ (respectively $F^{N}$) is a bijection between $(0,+\infty)$
and $(1,+\infty)$ (respectively, $(0,+\infty)$), hence, there exists a unique solution
if $r\delta>1$ (respectively, $r\delta>0$), which holds, in particular, for large $r$.
Furthermore, as both $\coth t$ and $\tanh t$ are bounded and tends to $1$
at $+\infty$,
it follows first that $k^{D/N}r\delta$ tends to $+\infty$ for large $r$, and then
that $k^{D/N}r\delta=r\delta+o(r\delta)$ for large $r$,
i.e. $k^{D/N}=1+o(1)$ and $k^{D/N}r\delta\to+\infty$, and \eqref{eq-kkd} give
$k^{D/N}=1+\cO(e^{-2r\delta})$
implying the estimates \eqref{eq-evd}.
Taking the derivative of \eqref{eq-kkd} with respect to $r$
we obtain
\begin{equation}
        \label{eq-kdpr}
(k^{D/N})'(r)=\frac{b(\rho-1)k^{D/N} r^{-\rho}}{\mp \cosh(k^{D/N}r\delta)\sinh(k^{D/N}r\delta)+r\delta}
=\cO(r^{-\infty}), \quad r\to+\infty.
\end{equation}
Recall that $C^{D/N}$ are normalization factors in \eqref{eq-forpsiDN}, we get
\[
C^{D/N}(r)^{-2}=G^{D/N}(2rk^{D/N}\delta) \delta,
\quad G^{D/N}(t):=\dfrac{1}{2}\Big(\dfrac{\sinh t}{t}\mp1\Big),
\]
which gives
\begin{align*}
\psi^{D}(t)&=\delta^{-1/2}G^{D}(2k^{D}r\delta)^{-1/2} \sinh(rkt),\\
\psi^{N}(t)&=\delta^{-1/2}G^{N}(2k^{N}r\delta)^{-1/2} \cosh(rkt).
\end{align*}
We have (we drop the indices $D/N$ when the expressions are the same):
\begin{gather*}
\partial_r(\delta^{-1/2})=\dfrac{\rho}{2r} \,\delta^{-1/2},\quad
\partial_r \big(G(2kr\delta)^{-1/2}\big)=-\dfrac{G'(2kr\delta)}{G(2kr\delta)}\,\partial_r(kr\delta)\cdot G(2kr\delta)^{-1/2},\\
\dfrac{(G^{D/N})'(t)}{G^{D/N}(t)}=\dfrac{ t \cosh t - \sinh t}{t(\sinh t\mp t)},
\end{gather*}
hence, in both cases
\[
\dfrac{G'(2kr\delta)}{G(2kr\delta)}=\cO(1), \quad r\to +\infty.
\]
Furthermore, using \eqref{eq-kdpr} (here again we drop the indices):
\[
\partial_r(kr\delta)=k'\cdot (r\delta) + k \cdot (r\delta)'=\cO(r^{-\rho}), \quad r\to +\infty.
\]
Therefore, 
$$
\left\{
\begin{aligned}
\partial_r \psi^{D}(t)=\cO(r^{-\rho}) \psi^{D} - C^{D}(r)k^{D}t \cosh\big(rk^{D}(\delta-t)\big)
\\
\partial_r \psi^{N}(t)=\cO(r^{-\rho}) \psi^{N} - C^{N}(r)k^{N}t \sinh\big(rk^{N}(\delta-t)\big)
\end{aligned}
\right. ,
$$
and
$$
\|\partial_r \psi^{D/N}\|^2_{L^2(0,\delta)}
\le \cO(r^{-2\rho})+ \cO(r^{-2\rho}) \dfrac{\sinh(2k^{D/N}r\delta)\pm2k^{D/N}r\delta}{\sinh(2k^{D/N}r\delta)\mp2k^{D/N}r\delta}=\cO(r^{-2\rho}).
$$
Finally
\[
\psi^{D/N}(0)^2=2r k^{D/N} \dfrac{\cosh(2k^{D/N}r\delta)\mp1}{\sinh(2k^{D/N}r\delta)\pm2rk^{D/N}\delta}
=2r\big(1+ \cO(r\delta e^{-2r\delta})\big)
\]
and 
$$\psi^{N}(\delta)^2=C^{N}(r)^{-2}= \dfrac{2k^{N}r}{\sinh(2k^{N}r\delta)+2k^{N}r\delta}=\cO(r^2\delta e^{-2r\delta}),$$
and the proposition is proved.

\end{document}